\documentclass{article}
\usepackage{preamble}

\usepackage[style=alphabetic, 
maxbibnames=5, backend=biber, 
doi=false, isbn=false, url=false]{biblatex}
\title{Segalification and the Boardman--Vogt tensor product}

\author{Shaul Barkan and Jan Steinebrunner}

\date{\today}

\begin{document}

\maketitle

\begin{abstract}
    We develop an analog of Dugger and Spivak's necklace formula, providing an explicit description of the Segal space generated by an arbitrary simplicial space.
    We apply this to obtain a formula for the Segalification of $n$-fold simplicial spaces, a new proof of the invariance of right fibrations,
    and a new construction of the Boardman--Vogt tensor product of $\infty$-operads, for which we also derive an explicit formula.
\end{abstract}

\tableofcontents

\section*{Introduction}

The nerve functor $\xN_\bullet\colon \mrm{Cat}_1 \to \sSet$ has a left adjoint, which assigns to a simplicial set $X$ its homotopy category $\mrm{ho}(X)$.
The homotopy category $\mrm{ho}(X)$ has as objects 
the $0$-simplices of $X$, and its morphisms are generated by the $1$-simplices of $X$ modulo the relations imposed by the $2$-simplices.
In the setting of $\infty$-categories, the nerve $\xN_\bullet \colon \Cat \to \PSh(\simp)$ is given by $\xN_n\calC \coloneq \Fun([n], \calC)^\simeq$
and participates in an adjunction
\[
    \mbm{C} \colon \PSh(\simp) \adj \Cat :\! \xN_\bullet
\]
where the left adjoint $\mbm{C}$ is left Kan extended from the inclusion $\simp \subset \Cat$.
The purpose of this note is to give a formula for $\mbm{C}(X)$.
\clearpage
The functor $\xN_\bullet$ is fully faithful and its essential image consists of the complete Segal spaces in the sense of Rezk \cite{rezk}.
Recall that a \hldef{Segal space}
is a simplicial space $X \colon \Dop \to \calS$ for which the natural map $X_n \to X_1 \times_{X_0} \cdots \times_{X_0} X_1$ is an equivalence for all $n$.
A Segal space is \hldef{complete} if the map $s_0\colon X_0 \to X_1$ induces an equivalence onto a certain union of components $X_1^{\rm eq} \subset X_1$.
Letting $\PSh_\CSS(\simp) \subset \PSh_\seg(\simp) \subset \PSh(\simp)$ denote the full subcategories of (complete) Segal spaces,
we can factor the adjunction $\mbm{C} \dashv \xN_\bullet$ as
\[\begin{tikzcd}
	\Cat & {\PSh_\CSS(\simp)} & {\PSh_\seg(\simp)} & {\PSh(\simp)}
	\arrow["{\xN_\bullet}", from=1-1, to=1-2]
	\arrow[""{name=0, anchor=center, inner sep=0}, shift right=2, hook, from=1-2, to=1-3]
	\arrow[""{name=1, anchor=center, inner sep=0}, shift right=2, hook, from=1-3, to=1-4]
	\arrow[""{name=2, anchor=center, inner sep=0}, "{\Lcpl}"', shift right=1, from=1-3, to=1-2]
	\arrow[""{name=3, anchor=center, inner sep=0}, "{\Lseg}"', shift right=1, from=1-4, to=1-3]
	\arrow["\simeq"', draw=none, from=1-1, to=1-2]
	\arrow["\dashv"{anchor=center, rotate=-90}, draw=none, from=2, to=0]
	\arrow["\dashv"{anchor=center, rotate=-90}, draw=none, from=3, to=1]
\end{tikzcd}\]
In his foundational work on complete Segal spaces \cite{rezk},
Rezk provides a formula for the Rezk-completion functor $\Lcpl$.
The purpose of this note is to provide a formula for the ``Segalification'' functor $\Lseg$.
Combining the two, one obtains an explicit description of the \category{} generated by an arbitrary simplicial space.

\subsubsection{Necklaces and Segalification}

Our formula for $\Lseg$ is heavily influenced by the work of Dugger and Spivak \cite{Dugger-Spivak} on the rigidification of quasicategories. 
It will involve a colimit indexed by a certain category of ``necklaces'' \cite[\S 3]{Dugger-Spivak}, which we now recall.
A \hldef{necklace} is a simplicial set 
$N = \Delta^{n_1}\vee \cdots \vee \Delta^{n_k}$
obtained by joining standard simplices at their start- and endpoints, as indicated in \cref{fig:necklace}.
\begin{figure}[h]
\centering
\tikzstyle{every node}=[circle, draw, fill=black,
                        inner sep=0pt, minimum width=2.5pt]
\begin{tikzpicture}[
    scale = .9,
    auto,
    ver/.style={circle,fill=white}]
    \filldraw[fill=black!20!white, draw=black]
    (0.2, 0) node {} -- (.4, .6) node {} -- (1.2,-.2) -- cycle;
    \filldraw[fill=black!20!white, draw=black]
    (1.2, -.2) -- (3.7,.3) -- (3, -.4) -- cycle;
    \draw[black] (3, -.4) node {} -- (3.7, .3) node {};
    \node at (1.2,-.2) {};
    \draw[black] (3.7, .3) node{} -- (4.7,-.3) node{} -- (5.4, .5) node{};

    \draw[->] (6,.2) to node[fill=white, draw=white, above=3pt] {$\lambda$} +(1,0);

    \begin{scope}[xshift = 7.5cm]
    \filldraw[fill=black!20!white, draw=black]
    (0.2, 0) node {} -- (1.2,-.2);
    \filldraw[fill=black!20!white, draw=black]
    (1.2, -.2) -- (2,.7) -- (3.7, .3) -- (3, -.4) -- cycle;
    \draw[black, dashed] (1.2,-.2) -- (3.7, .3);
    \draw[black] (2,.7) node{} -- (3, -.4) node{};
    \node at (1.2,-.2) {};
    \node at (3.7,.3) {};
    \filldraw[fill=black!20!white, draw=black]
    (3.7, .3) node{} -- (4.7,-.3) node{} -- (5.4, .5) node{} -- cycle;
    \end{scope}
\end{tikzpicture}
\caption{A morphism of necklaces $\lambda\colon \Delta^2 \vee \Delta^2 \vee \Delta^1 \vee \Delta^1 \to \Delta^1 \vee \Delta^3 \vee \Delta^2$
defined by collapsing the first edge and including the remaining necklace as indicated.}
\label{fig:necklace}
\end{figure}
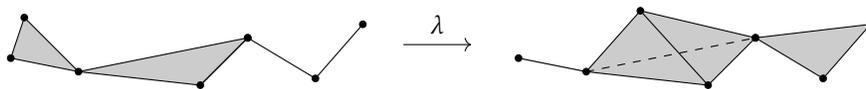
Following Dugger and Spivak, we define $\hldef{\Nec}$ to be the (non-full) subcategory of $\sSet$ whose objects are necklaces and whose morphisms are maps of simplicial sets $f \colon \Delta^{n_1}\vee \cdots \vee \Delta^{n_k} \to \Delta^{m_1}\vee \cdots \vee \Delta^{m_l}$ that preserve the minimal and maximal elements.

We can now state the formula for $\Lseg$ in terms of necklaces.
For the sake of simplicity, we state the formula here only for the case of $1$-simplices $\Lseg(X)_1$.
This suffices to determine $\Lseg(X)_n$ for all $n$ by the Segal condition.
\begin{thmA}\label{thmA:necklace-formula}
    For every simplicial space $X \in \PSh(\simp)$ 
    there is a canonical equivalence:
    \begin{align*}
    \Lseg(X)_1 
    & \simeq \colim_{N \in \Nec^\op} \Map_{\PSh(\simp)}(N, X) \\
    & \simeq 
    \colim_{\Delta^{n_1}\vee \cdots \vee \Delta^{n_r} \in \Nec^\op} X_{n_1} \times_{X_0} \cdots \times_{X_0} X_{n_r} 
    \end{align*} 
\end{thmA} 
This indeed generalises the formula for the homotopy category of a simplicial set $\mrm{ho}(X)$ mentioned above, as we shall see in \cref{ex:ho-of-X}.
\cref{thmA:necklace-formula} can also be deduced from the results of Dugger--Spivak by passing through the various model categories for \categories{},
but we will instead give a ``synthetic'' proof, as we believe it to be insightful.
A formula for monoidification analogous to \cref{thmA:necklace-formula} was obtained in \cite{yuan}.

\subsubsection{Application: right fibrations}
Our first application is to the notion of right fibrations of simplicial spaces in the sense of Rezk (see \cite[Remark 3.1]{Rasekh17}). 
Using the Segalification formula we show that right fibrations of simplicial spaces are invariant under $\bbL_{\CSS}$-equivalences, i.e.~if $f \colon X \to Y$ is a map of simplicial spaces such that $\Lcss(f)$ is an equivalence, then base change along $f$ induces an equivalence $f^\ast \colon \PSh(\simp)_{/Y}^{\rm r-fib} \simeq \PSh(\simp)_{/Y}^{\rm r-fib}$.
This implies that right fibrations over an arbitrary simplicial space $X$ model presheaves on the associated \category{} $\cat{X}$:

\begin{corA}[Rasekh]
    For any simplicial space $X$ the functor $\cat{-}$ induces an equivalence
    \[
        \PSh(\Delta)_{/X}^{\rm r-fib} 
        \xrightarrow[\simeq]{\cat{-}}
        \Catover{\cat{X}}^{\rm r-fib}
        \simeq
        \PSh(\cat{X})
    \] 
\end{corA}

A model-categorical version of this result was previously proven by Rasekh \cite[Theorem 4.18 and 5.1]{Rasekh17}.
Our proof has the advantage of being ``synthetic'' and also significantly shorter.
An alternative formulation of this corollary is to say that (the nerve of) the universal right fibration 
$\calS_*^\op \to \calS^\op$
classifies right fibrations of arbitrary simplicial spaces.%
\footnote{
    Note that the difficult part of this statement is the existence of a universal right fibration of simplicial spaces.
    Once existence is shown, it follows formally that the universal right fibration for simplicial space must agree with the one for complete Segal spaces.
}

We shall now use the above result to give a formula for $\cat{-} \colon \PSh(\simp) \to \Cat$.
Let us write $\simp_{\max} \subseteq \simp$ for the wide subcategory spanned by morphisms which preserve the maximal element,
and given a simplicial space $X \colon \Dop \to \calS$ let us write $p_X \colon \simp_{/X} \to \simp$ for the associated right fibration.

\begin{corA}
    For every simplicial space $X \in \PSh(\simp)$,
    the last vertex functor $\simp_{/X} \to \cat{X}$ 
    induces an equivalence of \categories{}
    \[\simp_{/X}[W_X^{-1}] \simeq \cat{X},\]
    where $W_X \coloneq p_X^{-1}(\simp_{\max}) \subseteq \simp_{/X}$.
\end{corA}

In the case that $X$ is level-wise discrete, i.e.~a simplicial set, this recovers a result of Stevenson \cite[Theorem 3]{stevenson} who attributes it to Joyal \cite[\S13.6]{Joyal-quasi}.
A synthetic proof for the case that $X$ is the nerve of an \category{} was given by Haugseng \cite[Proposition 2.12]{haugseng-coend}, and extended to arbitrary $X$ by Hebestreit--Steinebrunner \cite[Corollary 3.8]{hebestreit-steinbrunner}, motivated by the present paper.

\subsubsection{Application: $(\infty,n)$-categories}

The $\infty$-category of $(\infty,n)$-categories admits many equivalent descriptions including
Rezk's complete Segal $\Theta_n$-spaces \cite{rezk2010cartesian} and Barwick's complete $n$-fold Segal spaces \cite{Barwick-nfold}.
These were shown to be equivalent by Barwick--Schommer-Pries \cite{barwick2011unicity} and later, using different techniques, also by Bergner--Rezk \cite{bergner2020comparison} and Haugseng \cite{haugseng2018equivalence}.
We shall now present an application of our main result to $n$-fold Segal spaces.

We say that an $n$-fold simplicial space $X \colon (\Dop)^{\times n} \to \calS$ is \hldef{reduced} if each of the $(n-k-1)$-fold simplicial spaces $X_{m_1,\dots,m_k,0,\bullet,\dots, \bullet}$ is constant.
We write $\hldef{\PSh^r(\simp^{\times n})} \subseteq \PSh(\simp^{\times n})$ for the full subcategory of reduced $n$-fold simplicial spaces
and let $\hldef{\Seg_{\Dop}^{n\rm-fold}} \subseteq \PSh^r(\simp^{\times n})$
denote the full subcategory spanned by the \hldef{$n$-fold Segal spaces}, i.e.~those that satisfy the Segal condition in each coordinate.
This inclusion $\Seg_{\Dop}^{n\rm-fold} \hookrightarrow \PSh^r(\simp^{\times n})$ admits a left adjoint and we give a formula for it:
\begin{thmA}
    The left adjoint $\bbL\colon \PSh^r(\simp^{\times n}) \to \Seg_{\Dop}^{n\rm-fold}$ may be computed as 
    $\bbL = \bbL_n \circ \dots \circ \bbL_1$
    where $\bbL_j\colon \PSh(\simp^{\times n}) \to \PSh(\simp^{\times n})$ denotes the endofunctor that Segalifies the $j$th coordinate:
    \begin{align*}
    (\bbL_j X)_{m_1, \dots, m_{j-1}, 1, m_{j+1}, \dots, m_n} 
    & \simeq 
    \colim_{\Delta^{n_1}\vee \cdots \vee \Delta^{n_r} \in \Nec^\op} 
    X_{m_1, \dots, n_1, \dots, m_n} 
    \underset{X_{m_1, \dots, 0, \dots, m_n}}{\times}
    \cdots
    \underset{X_{m_1, \dots, 0, \dots, m_n}}{\times}
    X_{m_1, \dots, n_r, \dots, m_n} .
    \end{align*} 
\end{thmA}

Note that the order of the $\bbL_j$ is crucial: we need to first Segalify $1$-morphisms, then $2$-morphisms, etc.
If one were to apply $\bbL_2$ and then $\bbL_1$ the result would not necessarily satisfy the Segal condition in the second coordinate.

\subsubsection{Application: the Boardman--Vogt tensor product}
One of the many great achievements of Lurie's book project on Higher Algebra \cite{HA} is the construction of a homotopy coherent symmetric monoidal structure $\otimes_{\rm Lurie}$ on the \category{} of \operads{} $\Op$, generalizing the Boardman-Vogt tensor product \cite[\S II.3]{boardman2006homotopy}. 
The defining property of $\calO \otimes_{\rm Lurie} \calP$ is that algebras over it are ``$\calO$-algebras in $\calP$-algebras'', i.e.~that for any symmetric monoidal \category{} $\calC$ there is an equivalence
\[
    \Alg_{\calO \otimes_{\rm Lurie} \calP}(\calC) \simeq 
    \Alg_\calO(\Alg_\calP(\calC)).
\]
The construction of $\otimes_{\rm Lurie}$ is quite intricate as it involves a delicate mix of quasicategorical and model categorical techniques.
We shall now describe how the necklace formula can be used to justify a simpler, alternative approach to the tensor product of \operads{}.
\begin{thmA}\label{thmA:BV-tensor}
    The tensor product of symmetric monoidal \categories{} uniquely restricts to a tensor product $\otimes_{\rm BV}$ on $\Op$ such that the envelope
    $\Env \colon (\Op, \otimes_{\rm BV}) \to (\SM, \otimes)$ is a symmetric monoidal functor.
    Moreover, for any two \operads{} $\calO$ and $\calP$ there is a canonical equivalence $\calO \otimes_{\rm BV} \calP \simeq \calO \otimes_\Lurie \calP$.
\end{thmA}

\begin{rem*}
Note that \cref{thmA:BV-tensor} does \textit{not} claim that $(\Op,\otimes_{\rm BV})$ and $(\Op,\otimes_{\Lurie})$ are equivalent as symmetric monoidal \categories{}.
It does, however, reduce the question to whether the envelope can be constructed as a symmetric monoidal functor $(\Op,\otimes_{\Lurie}) \dashrightarrow (\SM,\otimes)$.
This is not entirely clear since the higher coherence data (associator, braiding, etc\dots) of Lurie's tensor product $\otimes_{\Lurie}$ is tricky to access.%
\footnote{
Lurie does give a model categorical construction of a (non-symmetric) monoidal structure which does have a recognizable universal property as a certain localization of $\infty$-categories over $\Fin_\ast$. The symmetric monoidal structure however is constructed by hand, and apart from the binary operation the relation between the two is not commented on.} 
Moreover, the authors are unaware of any applications in which the specific coherence of Lurie's construction plays a role.
\end{rem*}

A novel consequence of \cref{thmA:BV-tensor} is that, at least in principle, the Boardman-Vogt tensor product is only as difficult to compute as necklace colimits.
The resulting formula will be easiest to express in the language of symmetric sequences.

\subsubsection{Outlook: symmetric sequences}
A \hldef{symmetric sequence} is a presheaf on the category of finite sets and bijections.
The disjoint union $\hldef{\amalg}$ and the product $\hldef{\times}$ of finite sets extend via Day convolution to symmetric monoidal structures on symmetric sequences $\hldef{\SymSeq} \coloneq \PSh(\Fin^\cong)$ which we respectively denote by $\hldef{\otimes}$ and $\hldef{\boxtimes}$.
Since $(\SymSeq,\otimes)$ is the free presentably symmetric monoidal \category{} on a single generator $\underline{\unit} \in \SymSeq$, evaluation induces an equivalence
\[
    \ev_{\underline{\unit}}\colon \Fun_{\CAlg(\PrL)}\left((\SymSeq,\otimes), (\SymSeq,\otimes)\right) \xrightarrow{\simeq} \SymSeq
\]
which endows $\SymSeq$ with yet another (non-symmetric) monoidal structure $\hldef{\circ}$ coming from the composition of endofunctors on the left side.
It is generally expected that $1$-colored (non-complete) \operads{} are equivalent to associative algebras for $\circ$ in $\SymSeq$,
and for a different definition of $\circ$ this was shown in \cite{Rune-symmetric-sequences}.
Given two such algebras $\calO, \calP \in \Alg_{\mathbb{E}_1}(\SymSeq, \circ)$ the necklace formula in this setting gives:
    \[ 
    \calO \otimes_{\rm BV} \calP \simeq 
    \colim_{\Delta^{n_1}\vee \cdots \vee \Delta^{n_r} \in \Nec^\op}  
    \left(\calO^{\circ n_1} \boxtimes \calP^{\circ n_1} \right) \circ  \cdots \circ \left(\calO^{\circ n_r} \boxtimes \calP^{\circ n_r}\right)
    \]
A formal proof of this does not fit within the scope of the present paper, 
as it relies on a good interface between \operads{} and symmetric sequences.
We plan to provide such an interface and provide a full proof of this in forthcoming work where we revisit the composition product of symmetric sequences from the perspective of equifibered theory.

\subsection*{Acknowledgments}

We would like to thank Rune Haugseng for helpful comments on an earlier draft, Manuel Krannich for pointing out an issue in an earlier proof of \cref{thmA:BV-tensor}, Maxime Ramzi for useful conversations related to this paper, and the anonymous referee for a detailed and helpful report.

The first author would like to thank the Hausdorff Research Institute for Mathematics for their hospitality during the  fall 2022 trimester program, funded by the Deutsche Forschungsgemeinschaft (DFG, German Research Foundation) under Germany's Excellence Strategy – EXC-2047/1 – 390685813.
The second author is supported by the ERC grant no.~772960, and would like to thank the Copenhagen Centre for Geometry and Topology (DNRF151) for their hospitality.

\section{Segalification}
\subsection{Necklace contexts}

Let us fix an \category{} $\calC$, which will be $\Dop$ in later sections.
In this section, we study the general problem of approximating a reflective localization functor $\bbL \colon \PSh(\calC) \to \PSh(\calC)$ in the sense of \cite[Proposition 5.2.7.4]{HTT} using a suitable auxiliary subcategory of $\PSh(\calC)$.

\begin{defn}
    A presheaf $X \in \PSh(\calC)$ is called \hldef{$\bbL$-local} if the unit map $X \to \bbL(X)$ is an equivalence, i.e.~if $X$ lies in the essential image of $\bbL$.
    A morphism of presheaves $f\colon Y \to Z \in \PSh(\calC)$ is called an \hldef{$\bbL$-local equivalence} if $\bbL(f)$ is an equivalence.
\end{defn}

\begin{defn}\label{defn:necklace-context}
    A \hldef{necklace context} is a triple $(\calC,\bbL,\calN)$ where $\calC$ and $\bbL$ are as above
    and $\calN \subseteq \PSh(\calC)$ is a full subcategory such that
    \begin{enumerate}
        \item 
        $\Yo_c \coloneq \Map_\calC(-, c)$ is $\bbL$-local for all $c\in \calC$.
        \item 
        $\Yo_c \in \calN$ for all $c \in \calC$
        and
        $\bbL(N)$ is representable for all $N \in \calN$.
    \end{enumerate}
\end{defn}

\begin{example}
    A compatible necklace category for a pair $(\calC,\bbL)$ as in \cref{defn:necklace-context} exists if and only if the first condition holds. 
    In this case, the minimal possible necklace category is given by the representable presheaves $\calN_{\min} \coloneq \Yo(\calC) \subseteq \PSh(\calC)$ and the maximal choice is given by $\calN_{\max}\coloneq \bbL^{-1}(\calC) \subseteq \PSh(\calC)$, namely all $X \in \PSh(\calC)$ such that $\bbL(X) \in \Yo(\calC)$. 
    The full subcategory $\calN_{\rm sub} \subseteq \PSh(\calC)$ spanned by all subobjects $A \subseteq \Yo_c$ such that $\bbL(A) \simeq \Yo_c$ is another possible choice.
\end{example}

Given a necklace context $(\calC,\bbL,\calN)$, the Yoneda embedding $\Yo \colon \calC \hookrightarrow \PSh(\calC)$ lands in $\calN \subseteq \PSh(\calC)$ and thus gives rise to an adjunction
\[\begin{tikzcd}
	{\hldef{\ell} \coloneq \bbL_{|\calN} \colon \calN} && {\calC  :\! \Yo =:\! \hldef{i}}
	\arrow[""{name=0, anchor=center, inner sep=0}, shift left=2, hook', from=1-3, to=1-1]
	\arrow[""{name=1, anchor=center, inner sep=0}, shift left=2, from=1-1, to=1-3]
	\arrow["\dashv"{anchor=center, rotate=-90}, draw=none, from=1, to=0]
\end{tikzcd}\]
Passing to to presheaves we obtain a quadruple adjunction
\[\begin{tikzcd}
	{\PSh(\calC)} &&& {\PSh(\calN).}
	\arrow[""{name=0, anchor=center, inner sep=0}, "{i_! = \ell^*}"{description}, shift left=3, shorten <=8pt, shorten >=8pt, hook, from=1-1, to=1-4]
	\arrow[""{name=1, anchor=center, inner sep=0}, "{i^\ast = \ell_\ast}"{description}, shift left=3, shorten <=8pt, shorten >=8pt, from=1-4, to=1-1]
	\arrow[""{name=2, anchor=center, inner sep=0}, "{\ell_!}"{description}, shift right=1, curve={height=18pt}, from=1-4, to=1-1]
	\arrow[""{name=3, anchor=center, inner sep=0}, "{i_\ast}"{description}, shift right=1, curve={height=18pt}, hook, from=1-1, to=1-4]
	\arrow["\dashv"{anchor=center, rotate=-90}, draw=none, from=2, to=0]
	\arrow["\dashv"{anchor=center, rotate=-90}, draw=none, from=0, to=1]
	\arrow["\dashv"{anchor=center, rotate=-90}, draw=none, from=1, to=3]
\end{tikzcd}\]

\begin{lem}\label{lem:alpha-and-beta}
    The natural transformation $\hldef{\beta \colon i^\ast \to \ell_!}$
    defined by
        \[
            \left(
            \beta \colon i^\ast \xrightarrow{i^\ast \circ u} 
            i^\ast \circ (\ell^\ast \circ \ell_!) \simeq 
            (\ell \circ i)^\ast \circ \ell_! \simeq
            \ell_!
            \right)
            \in \Fun( \PSh(\calN), \PSh(\calC)).
        \] 
    is an $\bbL$-local equivalence.
\end{lem}
\begin{proof}
    The source and target of $\bbL(\beta) \colon \bbL i^\ast \to \bbL \ell_!$ are both left adjoints,
    when thought of as functors $\PSh(\calN) \to \PSh_{\bbL-\mrm{loc}}(\calC)$ so it suffices to check that $i^\ast (u) \colon i^\ast \to i^\ast (\ell^\ast \ell_!)$ evaluates to an $\bbL$-local equivalence at representable presheaves. 
    To see this, note the adjunction $\ell_! \dashv \ell^\ast$ agrees with $\ell \dashv i$ on representables and thus
    $u|_\calN$ is the unit $\Id_\calN \to i \ell = \bbL$.
    Finally, we apply $i^*\colon \calN \subset \PSh(\calN) \to \PSh(\calC)$, which is simply the inclusion $\calN \subset \PSh(\calC)$.
    Therefore the restriction of $i^*(u)$ to representables is the canonical map $N \to \bbL(N)$ for all $N \in \calN$.
    \qedhere
\end{proof}

\begin{defn}
    Given a necklace category $(\calC,\bbL,\calN)$ we define
    \[
       \hldef{Q_\calN}\colon  \PSh(\calC) \xrightarrow{i_*} \PSh(\calN) \xrightarrow{\ell_!} \PSh(\calC).
    \]
    This functor receives a canonical natural transformation from the identity 
    \[
        \hldef{\lambda}\colon  \Id_{\PSh(\calC)} \xleftarrow{\simeq} i^* \circ i_* 
        \xrightarrow{\beta \circ i_\ast} 
        \ell_! \circ i_* = Q_\calN.
    \]
\end{defn}

\begin{rem}\label{rem:formula-for-Q}
    The functor $i_\ast \colon \PSh(\calC) \to \PSh(\calN)$ may be computed as
    $(i_\ast X)(N) \simeq \Map_{\PSh(\calC)}(N, X)$.
    By the pointwise formula for left Kan extensions thus have
    \[
        Q_\calN(X)(c) = 
        \colim_{(N, \ell(N) \leftarrow c) \in (\calN \times_{\calC} \calC_{c/})^\op} \Map_{\PSh(\calC)}(N, X).
    \]
\end{rem}

Note that by \cref{lem:alpha-and-beta}, $\lambda\colon \id \to Q_\calN(X)$ is $\bbL$-local and thus
the unit transformation $\id \to \bbL$ factors through $\lambda$.
The resulting natural transformation $Q_\calN \to \bbL$ is then $\bbL$-local by cancellation. 
We thus conclude:

\begin{cor}\label{prop:lambda-L-local}
    There exists a canonical $\bbL$-local natural transformation $Q_\calN \to \bbL$ such that for any $X \in \PSh(\calC)$ the map $Q_\calN(X) \to \bbL(X)$ is an equivalence if and only if $Q_\calN(X)$ is $\bbL$-local.
\end{cor}

\subsection{Segalification}

We now specialize to the setting of Segal spaces, where the localization $\Lseg$ is defined as the left adjoint to the full inclusion
\[
    \Seg_{\Dop}(\calS) \hookrightarrow \Fun(\Dop, \calS)
\]
of those simplicial spaces $X_\bullet$ for which the map $X_n \to X_1 \times_{X_0} \dots \times_{X_0} X_1$ is an equivalence.
We will choose a necklace category and show that $Q_\calN \simeq \Lseg$.

\begin{rem}\label{rem:Dugger-Spivak}
The formula we will arrive at for $\Lseg$ is closely related to the work of Dugger--Spivak \cite{Dugger-Spivak},
who construct a functor
\[
    \mfr{C}^{\rm nec}\colon \mrm{sSet} \to \mrm{sCat}
\]
from the category of simplicial sets to the category of simplicial categories, which they show to be weakly equivalent to the left adjoint $\mfr{C}$ of the coherent nerve.
This gives a formula for the mapping spaces in $\mfr{C}(Z_\bullet)$ (as a colimit over the necklace category) when $Z_\bullet$ is a simplicial set in the Joyal model structure.
The case of a simplicial space follows by using the left Quillen equivalence $t_!\colon \mrm{ssSet} \to \mrm{sSet}$ constructed by Joyal--Tierney \cite{JT06}.
\end{rem}

\subsubsection{Segal spaces}

Let us recall the category of necklaces, introduced by Dugger and Spivak \cite{Dugger-Spivak}.

\begin{defn}
    The \hldef{concatenation $A \vee B$} of two bi-pointed simplicial sets $(A, a_{\min}, a_{\max})$ and $(B, b_{\min}, b_{\max})$ is defined as the pushout
    \[\begin{tikzcd}
    	{\Delta^0} & B \\
    	A & {\hldef{A \vee B},}
    	\arrow["{a_{\max}}"', from=1-1, to=2-1]
    	\arrow[from=1-2, to=2-2]
    	\arrow[from=2-1, to=2-2]
    	\arrow["{b_{\min}}", from=1-1, to=1-2]
    	\arrow["\lrcorner"{anchor=center, pos=0.125, rotate=180}, draw=none, from=2-2, to=1-1]
    \end{tikzcd}\]
    which we point as $(A \vee B, a_{\min}, b_{\max})$.
    This defines a (non-symmetric) monoidal structure on the category of bi-pointed simplicial sets.
\end{defn}

\begin{defn}
    A \hldef{necklace} is a bi-pointed simplicial set obtained by concatenating simplices $(\Delta^n, 0, n)$,
    i.e.~it is of the form $N = \Delta^{n_1} \vee \dots \vee \Delta^{n_k}$.
    We let $\hldef{\Nec}$ denote the category whose objects are necklaces and whose morphisms are maps of bi-pointed simplicial sets.
\end{defn}

While the category $\Nec$ will play the central role in the Segalification formula, we will need a slightly bigger category to set up the necklace context.
\begin{defn}
    Let $\hldef{\calN} \subset \PSh(\simp)$ denote the essential image of the faithful functor $\Nec \to \PSh(\simp)$.
\end{defn}

\begin{lem}\label{lem:necklace-context}
    Segalification $\Lseg \colon \PSh(\simp) \to \Seg_{\Dop}(\calS)$ restricts to a functor
    $\Lseg|_{\calN} \colon \calN \to \simp$.
    In particular, the triple $(\simp, \Lseg, \calN)$ is a necklace context.
\end{lem}
\begin{proof}
    We will show by induction on $n$ that the inclusion $\Delta^{\{0,\dots,n_1\}} \vee \dots \vee \Delta^{\{n_k,\dots,n\}} \hookrightarrow \Delta^n$ is a Segal equivalence, thereby proving the claim.
    Consider the nested inclusion
    \[\Delta^{\{0,1\}}\vee \dots \vee \Delta^{\{n-1,n\}} \hookrightarrow \Delta^{\{0,\dots,n_1\}} \vee \dots \vee \Delta^{\{n_k,\dots,n\}} \hookrightarrow \Delta^n.\]
    The composite is a Segal equivalence by definition, 
    and since $\bbL$ preserves colimits and $n_{j+1}-n_j<n$ the first map is a Segal equivalence by the induction hypothesis.
    The second map is therefore a Segal equivalence by cancellation.
\end{proof}

\begin{rem}\label{rem:L-faithful}
    Note that the map $N \to \Lseg(N) = \Delta^n$ is a monomorphism for each necklace $N$.
    In particular, for any two necklaces $N, M$, the map
    \[
        \Map_\calN(N, M) \too \Map_{\PSh(\simp)}(\Lseg(N), \Lseg(M)) \simeq \Map_{\simp}([n], [m])
    \]
    is a monomorphism, i.e. $\bbL\colon \calN \to \simp$ is faithful.
\end{rem}

\subsubsection{The Segal condition}

Since $(\simp, \Lseg, \calN)$ is a necklace context we have by \cref{lem:necklace-context} a functor
\[
    Q\colon \PSh(\simp) \xrightarrow{i_*} \PSh(\calN) \xrightarrow{\ell_!} \PSh(\simp).
\]
By \cref{prop:lambda-L-local} this comes with an $\Lseg$-local natural transformation $Q \to \Lseg$.
We may compute the functor $Q(-)$ using \cref{rem:formula-for-Q} as:
    \[
        Q(X)_n = 
        \colim_{(N, l(N) \leftarrow [n]) \in (\calN \times_{\simp} \simp_{[n]/})^\op} \Map_{\PSh(\simp)}(N, X).
    \]
Below we show that $Q(X)$ is always a Segal space,
and thus by \cref{prop:lambda-L-local} that $Q  \simeq \Lseg$.

\begin{defn}\label{defn:necklace-decomposition}
    For any necklace $N$ we let $\hldef{\iota_N} \colon \Delta^1 \to \Lseg(N)$ denote the unique map that preserves the extrema.
    Given $[n] \in \simp$ we define the functor
    \[
        \xJoin \colon \prod_{i=1}^n \Nec 
        \hookrightarrow \calN \times_{\simp} \simp_{[n]/}  ,
    \]
    by joining necklaces at their endpoints
    \[(M_1,\dots, M_n) \longmapsto \left(M_1 \vee \cdots \vee M_n,[n] \xrightarrow{\Lseg(\iota_1 \vee \cdots \vee \iota_n)} 
    \Lseg(\Lseg(M_1) \vee \dots \vee \Lseg(M_n)) 
    \simeq \Lseg(M_1 \vee \dots \vee M_n)
    \right). \]
\end{defn}

\begin{lem}\label{lem:necklace-decomposition}
    The functor $\xJoin$ is fully faithful and admits a right adjoint.
    In particular it is initial.
\end{lem}
\begin{proof}
    Fully-faithfulness follows by unwinding definitions.~We claim that a right adjoint to $J$ is given by the following
    \[J^R \colon (N,\alpha \colon [n] \to \Lseg(N)) \longmapsto (N_{\alpha(0),\alpha(1)}, \dots, N_{\alpha(n-1),\alpha(n)} ) \]
    where $N_{\alpha(j),\alpha(j+1)} \coloneq N \cap \Delta^{\{\alpha(j),\dots, \alpha(j+1)\}}$.
    To see this, note that a tuple of necklace morphisms
    \[(M_1,\dots,M_n) \to (N_{\alpha(0),\alpha(1)}, \dots ,N_{\alpha(n-1),\alpha(n)})
    = J^R(N, \alpha)
    \]
    is equivalent to a morphism
    $ 
    M_1 \vee \dots \vee M_n \to N_{\alpha(0),\alpha(1)} \vee \dots \vee N_{\alpha(n-1),\alpha(n)} \subset N
    $
    such that $M_i$ lands in $N_{\alpha(i-1), \alpha(i)}$.
    These can be identified with morphisms 
    $J(M_1,\dots,M_n) \to (N,\alpha)$ in $\calN \times_{\simp} \simp_{[n]/}$ and thus $J^R$ is indeed right adjoint to $J$.
\end{proof}

\begin{obs}\label{obs:rewriting-Q}
    The finality in \cref{lem:necklace-decomposition} implies that $Q(X)_n$ may be computed as:
    \begin{align*}
        Q(X)_n 
        &\simeq \colim_{(M_1,\dots, M_n) \in (\Nec^\op)^n} \Map_{\PSh(\simp)}(M_1 \vee \dots \vee M_n, X)\\
        &\simeq \colim_{(M_1,\dots, M_n) \in (\Nec^\op)^n} \Map_{\PSh(\simp)}(M_1, X) \underset{X_0}{\times} \dots \underset{X_0}{\times} 
        \Map_{\PSh(\simp)}(M_n, X)
    \end{align*}
    In particular for $n=0$ we just get $Q(X)_0 = X_0$.
    While this is a simplification of the general formula from \cref{rem:formula-for-Q}, it has the downside that the functoriality in $[n]$ is not clear in general.
    However, we can still see the functoriality in inert maps $\varphi\colon [m] \intto [n]$, as it is simply given by restricting to the $M_i$ that correspond to the image of $\varphi$.
    This functoriality will suffice to check the Segal condition.
\end{obs}

\begin{prop}\label{prop:Segal-one-iteration}
    For any simplicial space $X_\bullet$ the simplicial space $Q(X)_\bullet$ is a Segal space.
\end{prop}
\begin{proof}
    Consider the following diagram:
    \[
    \begin{tikzcd}
        \colim\limits_{(M_1, \dots, M_n) \in (\Nec^\op)^n} \Map(M_1, X) \underset{X_0}{\times} \dots \underset{X_0}{\times} 
        \Map(M_n, X) \ar[r] \ar[d] &
        Q(X)_n \ar[d] \\
        \colim\limits_{M_1 \in \Nec^\op} \Map(M_1, X) \underset{X_0}{\times} \dots \underset{X_0}{\times} \colim\limits_{M_n \in \Nec^\op} \Map(M_n, X) \ar[r] &
        Q(X)_1 \underset{Q(X)_0}{\times} \dots \underset{Q(X)_0}{\times} Q(X)_1.
    \end{tikzcd}
    \]
    The horizontal maps are equivalences by 
    \cref{lem:necklace-decomposition} and \cref{obs:rewriting-Q}.
    The left vertical map is an equivalence since the cartesian product in $\calS_{/X_0}$ preserves colimits in each variable.
\end{proof}

\subsection{Variations on the Segalification formula}
\subsubsection{A formula for mapping spaces}
Given a Segal space $X \in \Seg_\Dop(\calS)$ the mapping spaces in the associated \category{} $\cat{X}$ may be computed as
\[
    \Map_{\cat{X}}(x, y) \simeq \{x\} \times_{X_0} X_1 \times_{X_0} \{y\}
\]
for any $x, y \in X_0$.
Below we show how to use the results of the previous section to derive a formula for these mapping spaces when $X$ is an arbitrary simplicial space.
In the case where $X$ is a simplicial set this recovers the formula of Dugger--Spivak \cite{Dugger-Spivak}, which inspired our Segalification formula.

\begin{lem}\label{lem:formula-for-mapping-space}
    For any simplicial space $X\in \PSh(\simp)$ and $x,y \in X$ there is a canonical equivalence
    \[
        |\Nec_{/(X,x,y)}| \simeq \Map_{\cat{X}}(x, y),
    \]
    where $\Nec_{/(X,x,y)} \subseteq \left(\PSh(\simp)_{\Delta^0 \amalg \Delta^0/}\right)_{/(X, x, y)}$ denotes the full subcategory spanned by necklaces.
\end{lem}
\begin{proof}
    Interpreting $\Nec$ as a full subcategory of $\PSh(\simp)_{\Delta^0 \amalg \Delta^0/}$ by recording the minimal and maximal vertex,
    we can fit $\Nec_{/(X, x, y)}$ into a cartesian square:
    \[
    \begin{tikzcd}
        \Nec_{/(X,x,y)} \ar[r] \ar[d] &
        \Nec \times_{\PSh(\simp)} \PSh(\simp)_{/X} \ar[d] \\
        {\{(x,y)\}} \ar[r] & X_0 \times X_0.
    \end{tikzcd}
    \]
    The top right corner is a right fibration over $\Nec$ corresponding to the presheaf 
    \[
    \Map_{\PSh(\simp)}(-, X) \colon \Nec^\op \to \calS.
    \]
    The weak homotopy type of the top right corner is thus the colimit of this functor, which is precisely the definition of $Q(X)_1$.
    While the functor $|-|\colon \Cat \to \calS$ 
    does not generally preserve pullbacks, it does preserve those cartesian squares where the bottom arrow is a map of spaces (because $\Dop$-colimits in $\calS$ are stable under base change).
    For the square at hand we obtain
    \[|\Nec_{/(X,x,y)}| \simeq \{(x,y)\} \times_{X_0^{\times 2}} Q(X)_1 \simeq \Map_{\cat{X}}(x,y),\]
    where the second equivalence holds since the Rezk-completion of $Q(X) \simeq \Lseg(X)$
    is the nerve of $\cat{X}$.
\end{proof}

\begin{rem}
    Given three points $x,y,z \in X$ the monoidal structure $\vee$ on $\Nec$ yields a functor
    \[
        \vee \colon \Nec_{/(X,x,y)} \times \Nec_{/(X,y,z)}
            \too
        \Nec_{/(X,x,z)}.
    \]
    On weak homotopy types this yields the composition 
    $\Map_{\cat{X}}(x,y) \times \Map_{\cat{X}}(y,z) \to \Map_{\cat{X}}(x,z)$
    in $\cat{X}$.
    (As in \cite[Eqn.~1.2]{Dugger-Spivak}.)
    This can be seen by an argument similar to \cref{lem:formula-for-mapping-space} using the necklace formula for $Q_2X$.
\end{rem}

\subsubsection{$1$-categories}
Similar to how $\Delta^\op$-colimits in $1$-categories can be computed as reflexive coequalizers,
$\Nec^\op$ colimits in a $1$-category can be reduced to certain ``thin'' necklaces.

\begin{defn}
    We say that a necklace $N = \Delta^{n_1}\vee \dots \Delta^{n_r} \in \Nec$ is \hldef{thin}
    if $\sum_i n_i \le r+1$ and $n_i \ge 1$,
    in other words if it consists of $1$-simplices and at most one two-simplex.
    If $N$ consists only of $1$-simplices, we say that it is \hldef{very thin}.
    Let $\hldef{\Nec_{\rm thin}} \subset \Nec$ denote the full subcategory of thin necklaces.
\end{defn}

\begin{lem}\label{lem:thin-necklace}
    The full inclusion $\Nec_{\rm thin}^\op \hookrightarrow \Nec^\op$ is $1$-final,
    that is, for any functor $\Nec^\op \to \calC$ to a $1$-category $\calC$ the colimit may equivalently be computed over $\Nec_{\rm thin}^\op$.
\end{lem}
\begin{proof}
    We need to show that for any necklace $N = \Delta^{n_1} \vee \dots \vee \Delta^{n_r} \in \Nec$ the slice category $\Nec_{{\rm thin}/N}$ is \emph{connected}.
    We enumerate the vertices of $N$ in their canonical order as $0, \dots, n = \sum_i n_i$.
    A very thin necklace over $N$, i.e.~a map $\Delta^1 \vee \dots \vee \Delta^1 \to N$, may equivalently be encoded as a non-decreasing path $0=a_0 \le \dots \le a_k=n$ in $[n]$.
    These paths are subject to the condition that we never have strict inequalities $a_l < \sum_{i=1}^s n_i < a_{l+1}$ for any $s$ and $l$.
    Suppose that $p = (0= a_0 \le \dots \le a_k = n)$ is such a path and $s$ is such that $p' = (0=a_0 \le \dots \widehat{a_s} \dots \le a_k = n)$ is still an admissible path.
    Then there is a thin necklace $M$ with a two-simplex $(a_{s-1} \le a_s \le a_{s+1})$ that contains both of these paths.
    In particular, the paths are connected through a zig-zag $p \to M \leftarrow p'$ as objects of $\Nec_{{\rm thin}/N}$.
    Proceeding by removing a vertex whenever possible, we see that every very thin necklace over $N$ is connected in $\Nec_{{\rm thin}/N}$ to a very thin necklace that corresponds to a minimal path in $N$.
    But there is only one path in $N$ that is minimal with respect to removing vertices, namely $(0 \le n_1 \le \dots \le \sum_{i=1}^{r-1} n_i \le n)$.
    Therefore all the very thin objects in $\Nec_{{\rm thin}/N}$ are connected, and thus the category is connected as every (thin) necklace contains a very thin necklace.
\end{proof}

\begin{example}\label{ex:ho-of-X}
Suppose that $X_\bullet$ is a simplicial space and we want to compute the homotopy category $h_1(\cat{X})$.
For simplicity, let us assume that $X_n$ is discrete for all $n$.%
\footnote{
    This is not a very restrictive assumption. 
    Starting with a general simplicial space $Y_\bullet$ we may base change it along a $\pi_0$-surjective map $Z_0 \to Y_0$ to get a simplicial space $Z_n = (Z_0)^{n+1} \times_{Y_0^{n+1}} Y_n$ such that the resulting functor $\cat{Z_\bullet} \to \cat{Y_\bullet}$ will be an equivalence.
    Now we may choose $Z_0$ to be discrete and define $X_n \coloneq \pi_0(Z_n)$.
    Then $\cat{Z_\bullet} \to \cat{X_\bullet}$ induces an equivalence on homotopy categories.
}
Then the set of morphisms in $h_1(\cat{X}))$ is exactly $\pi_0 (\Lseg(X)_1)$ and we may compute it as the colimit
\[
    \mrm{Mor}(h_1(\cat{X})) 
    \cong \colim_{N \in \Nec^\op} \Map(N, X)
    \cong \colim_{\Delta^{n_1}\vee \cdots \vee \Delta^{n_r} \in \Nec^\op} X(\Delta^{n_1}) \times_{X(\Delta^0)} \cdots \times_{X(\Delta^0)} X(\Delta^{n_r}).
\]
in the $1$-category of sets.
By \cref{lem:thin-necklace} it suffices to take the colimit over $\Nec_{\rm thin}^\op$.
Moreover the very thin necklaces are $0$-final, so the colimit may be expressed as a coproduct over the very thin necklaces modulo an equivalence relation.
This leads to the formula
\[
    \mrm{Mor}(h_1(\cat{X})) 
    \cong 
    \left(\coprod_{n \ge 0} X_1 \times_{X_0} \dots \times_{X_0} X_1\right)/\sim
\]
where the equivalence relation is generated by
$(f_1, \dots, f_n) \sim (f_1, \dots, f_{i-1}, g, f_{i+2}, \dots, f_n)$
whenever there is a $2$-simplex in $X$ witnessing $f_{i+1} \circ f_i = g$,
and $(f_1, \dots, f_n) \sim (f_1, \dots, f_{i-1}, f_{i+1}, \dots, f_n)$ whenever $f_i$ is a degenerate $1$-simplex.
This recovers the classical formula for the homotopy category of a simplicial set:
namely, it is the free category on the edges of $X$ modulo the relations generated by the $2$-simplices and the degenerate $1$-simplices.
\end{example}

\subsubsection{Segalification in other categories}

In this section we establish criteria on a presentable \category{} which guarantee that Segalification is given by the necklace formula. 
This is summarized by the following result, which we prove in the remainder of this section.

\begin{prop}\label{prop:Segalification-general-cats}
    Let $\calV$ be a presentable \category{} in which sifted colimits are stable under base change.
    Then the left adjoint to the inclusion
    $\Seg_\Dop(\calV) \hookrightarrow \Fun(\Dop,\calV)$ 
    is given by the necklace formula:
    \[
        \Lseg(X)_1 
        \simeq \colim_{N \in \Nec^\op} (i_*X)(N)
        \simeq \colim_{\Delta^{n_1} \vee \dots \vee \Delta^{n_k} \in \Nec^\op} X_{n_1} \times_{X_0} \dots \times_{X_0} X_{n_k}.
    \]
    where $i_*$ denotes the right Kan extension
    $i_*\colon \Fun(\Dop, \calV) \to \Fun(\calN^\op, \calV)$.
\end{prop}

Recall that if $\frX$ is an $\infty$-topos then \textit{all} colimits in $\frX$ are stable under base change (i.e.~they are ``universal'') \cite[Proposition 6.1.3.19]{HTT}.
    In particular, \cref{prop:Segalification-general-cats} applies to $\infty$-topoi.
    A wider variety of examples is provided by passing to algebras over \operads{}.
\begin{example}\label{ex:O-algebra-universal-colimits}
    Let $\calV$ be a presentably symmetric monoidal \category{}%
    \footnote{In fact, it suffices to ask that the monoidal structure is compatible with sifted colimits.}
    and $\calO$ be an \operad{}.
    Then the forgetful functor $\Alg_\calO(\calV) \to \Fun(\mathrm{col}(\calO), \calV)$, which only remembers the object assigned to each colour $c \in \mathrm{col}(\calO)$, preserves and creates both limits and sifted colimits \cite[{}3.2.2.4 \& 3.2.3.1]{HA}.
    Consequently, if sifted colimits in $\calV$ are stable under base change, then the same holds for $\Alg_\calO(\calV)$. 
\end{example}

\begin{lem}\label{lem:Segalification-for-V}
    \cref{prop:Segalification-general-cats} holds if we assume that $\Nec^\op$-colimits in $\calV$ are stable under base change.
\end{lem}
\begin{proof}
    The left adjoint $\Lseg$ exists by the adjoint functor theorem.
    Since $\calV$ is presentable we may find a small \category{} $\calE$ and a fully faithful right adjoint $I\colon \calV \hookrightarrow \PSh(\calE)$.
    We denote the resulting adjunction on presheaf categories by
    \[I^{\simp} \colon \Fun(\Dop, \calV) \adj \Fun(\Dop, \PSh(\calE)) :\! L^{\simp}\]
    We now define an endofunctor $Q^\calV \coloneq L^{\simp} \circ Q \circ I^{\simp} \colon \Fun(\Dop, \calV) \to \Fun(\Dop, \calV)$ where $Q$ is the endofunctor on $\Fun(\Dop, \PSh(\calE)) \simeq \Fun(\calE, \PSh(\simp))$, pointwise given by the usual formula (see \cref{obs:rewriting-Q}).
    This $Q^\calV$ receives a natural transformation 
    \[
        \lambda^\calV\colon
        L^{\simp} \circ I^{\simp} 
        \xrightarrow{L^{\simp} \circ \lambda \circ I^{\simp}}
        L^{\simp} \circ Q \circ I^{\simp} = Q^\calV
    \]
    coming from $\lambda\colon \id \to Q$. (Note that the source of $\lambda$ is $L^{\simp} \circ I^{\simp} \simeq \id_{\Fun(\Dop, \calV)}$.)
    This transformation is a Segal equivalence. 
    Indeed,
    if $X, Y \colon \Dop \to \calV$ and $Y$ is Segal, then in the commutative square
\[\begin{tikzcd}[column sep = large]
	{\Map_{\Fun(\Dop, \calV)}(Q^\calV X, Y)} & {\Map_{\Fun(\Dop, \calV)}((L^{\simp} \circ I^{\simp})(X), Y)} \\
	{\Map_{\Fun(\Dop, \PSh(\calE))}((Q \circ I^{\simp})(X), I^{\simp}(Y))} & {\Map_{\Fun(\Dop, \PSh(\calE))}(I^{\simp}(X), I^{\simp}(Y))}
	\arrow["\simeq", from=1-2, to=2-2]
	\arrow["\simeq", from=1-1, to=2-1]
	\arrow["{(-) \circ \lambda^\calV_X}", from=1-1, to=1-2]
	\arrow["{(-) \circ \lambda_{I^{\simp}(X)}}", from=2-1, to=2-2]
\end{tikzcd}\]
    the bottom map is an equivalence since $I^{\simp}(Y)$ is Segal and thus so is the top map.

    It remains to show that $Q^\calV(X)$ is Segal for all $X \colon \Dop \to \calV$.
    This follows from the same proof as \cref{prop:Segal-one-iteration} by using that $\Nec^\op$-colimits are stable under base change.
\end{proof}

In principle, it might be difficult to tell whether $\Nec^\op$-shaped colimits are stable under base change in a given \category{}, but fortunately
$\Nec^\op$ is a sifted category, colimits over which are well understood.
We will deduce this from the following fact, to which it is intimately linked:

\begin{lem}\label{lem:Segalification-products}
    The Segalification functor $\bbL\colon \PSh(\Delta) \to \PSh(\Delta)$ preserves products.
\end{lem}
\begin{proof}
    The two functors
    \[
        \PSh(\Delta) \times \PSh(\Delta) \longrightarrow \Cat
    \]
    given by $(X,Y) \mapsto \bbL(X \times Y)$ and $\bbL(X) \times \bbL(Y)$ both preserve colimits in both variables.
    Therefore it suffices to check that the natural transformation between them is an equivalence on $(\Delta^n, \Delta^m)$.
    But in this case it is easy to check because $\Delta^n$, $\Delta^m$ and $\Delta^n \times \Delta^m$ are all Segal spaces.
\end{proof}

\begin{lem}\label{lem:nec-sifted}
    The category $\Nec^\op$ is sifted.
\end{lem}
\begin{proof}
    We need to show that the diagonal functor $\Delta\colon \Nec^\op \to \Nec^\op \times \Nec^\op$ is final.
    Equivalently, we need that for all $A, B \in \Nec$ the slice $\Nec \times_{\Nec^2} \Nec^2_{/(A,B)}$ is weakly contractible.
    This category is equivalent to  the full subcategory 
    $\Nec_{/A\times B} \subseteq \left(\PSh(\simp)_{\Delta^0 \amalg \Delta^0/}\right)_{/A\times B}$
    spanned by necklaces,
    where the product $A \times B$ is taken in the \category{} $\PSh(\simp)_{\Delta^0 \amalg \Delta^0/}$ of bipointed simplicial spaces.
    By \cref{lem:formula-for-mapping-space} the weak homotopy type of this category computes the mapping space
    \[
        |\Nec_{/(A\times B, (a_{\min}, b_{\min}),(a_{\max}, b_{\max}))}|
        \simeq
        \Map_{\cat{A \times B}}((a_{\min}, b_{\min}),(a_{\max}, b_{\max})).
    \]
    Since $\cat{-}$ commutes with products by \cref{lem:Segalification-products}, we may compute
    $\cat{A \times B} \simeq \cat{A} \times \cat{B} = [n] \times [m]$.
    In particular, we see that the mapping space $\Map_{[n] \times [m]}((0,0), (n,m))$ is contractible, which completes the proof.
\end{proof}
\section{Applications}
\subsection{Segalification and right fibrations}

Throughout this section we fix a presentable \category{} $\calV$ and a factorization system $(\calV^L,\calV^R)$.

\begin{defn}
    We say that $X\colon \Delta^\op \to \calV$ is \hldef{right-$\calV^R$-fibered} 
    if $d_0\colon X_n \to X_{n-1}$ is in $\calV^R$ for all $n \ge 1$.
\end{defn}

\begin{obs}\label{rem:R-fibered-Segal}
    A Segal object $X\colon \Dop \to \calV$
    is right-$\calV^R$-fibered if and only if $d_0\colon X_1 \to X_0$ is in $\calV^R$.
    Indeed, morphisms in $\calV^R$ are closed under pullbacks in the arrow category 
    \cite[Proposition 5.2.8.6.(8)]{HTT}
    and when $X$ is Segal we can write $d_0 \colon X_n \to X_{n-1}$ as a pullback in the arrow category of the following cospan 
    \[\begin{tikzcd}
	{X_{n-1}} && {X_0} && {X_1} \\
	{X_{n-1}} && {X_0} && {X_0.}
	\arrow["{d_1 \circ \cdots \circ d_n}", from=1-1, to=1-3]
	\arrow["{=}"', from=1-3, to=2-3]
	\arrow["{=}"', from=1-1, to=2-1]
	\arrow["{d_0}"', from=1-5, to=2-5]
	\arrow["{=}"', from=2-5, to=2-3]
	\arrow["{d_1 \circ \cdots \circ d_n}", from=2-1, to=2-3]
	\arrow["{d_0}"', from=1-5, to=1-3]
    \end{tikzcd}\]
\end{obs}
Under suitable assumptions the necklace formula can be used to show that Segalification preserves right-$\calV^R$-fibered objects.

\begin{prop}\label{cor:right-fibered-usable}
    Suppose $\calV$ and $(\calV^L,\calV^R)$ are such that:
    \begin{enumerate}
        \item 
        sifted colimits in $\calV$ are stable under base change and
        \item 
        The full subcategory $\calV^{(R)}_{/X} \subseteq \calV_{/X}$ on those $Y \to X$ that are in $\calV^R$ is closed under sifted colimits for all $X \in \calV$.
    \end{enumerate}    
    Then Segalification $\Lseg \colon \Fun(\Dop,\calV) \to \Seg_\Dop(\calV)$ preserves right-$\calV^R$-fibered objects.
\end{prop}
\begin{proof}
    By \cref{rem:R-fibered-Segal} it suffices to show that if $X\colon \Dop \to \calV$ is right-$\calV^R$-fibered then $d_0\colon \bbL(X)_1 \to \bbL(X)_0$ is in $\calV^R$.
    We claim that for every necklace $N = \Delta^{n_1} \vee \dots \vee \Delta^{n_k}$ the map $(i_*X)(N) \simeq X_{n_1} \times_{X_0} \dots \times_{X_0} X_{n_k} \to X_0$ 
    induced by the inclusion of the terminal vertex $\Delta^0 \to N$ is in $\calV^R$.
    Indeed, when $N$ is a simplex this holds by assumption and the general case follows by taking pullbacks since morphisms in $\calV^R$ are closed under base change and composition.
    
    Using the necklace formula (see \cref{prop:Segalification-general-cats}), we can write $d_0 \colon \bbL(X)_1 \to \bbL(X)_0$ as a colimit
    \[
        \bbL(X)_1 \simeq\colim_{N \in \Nec^\op} (i_*X)(N) 
        \too X_0 = \bbL(X)_0
    \]
    in $\calV_{/X_0}$, of a diagram indexed by $\Nec^\op$, of morphisms $(i_*X)(N) \to X_0$ that lie in $\calV^{(R)}_{/X_0}$.
    Since $\Nec^\op$ is sifted (see \cref{lem:nec-sifted}) and $\calV^{(R)}_{/X_0} \subseteq \calV_{/X_0}$ is closed under sifted colimits, the colimit lies in $\calV^{(R)}_{/X_0}$.
\end{proof}

\subsubsection{Right fibrations of simplicial spaces}
We shall now apply \cref{cor:right-fibered-usable} to right fibrations of simplicial spaces in the sense of Rezk, whose definition we briefly recall. 
\begin{defn}
    A map of simplicial spaces $p\colon  X \to Y$ is called a \hldef{right fibration} if the square
    \[
        \begin{tikzcd}
            X_n \ar[r, "d_0"] \ar[d, "p"'] & 
            X_{n-1} \ar[d, "p"] \\
            Y_n \ar[r, "d_0"] & 
            Y_{n-1}
        \end{tikzcd}
    \]
    is cartesian for all $n \ge 1$.
\end{defn}

The Segalification formula implies the following:
\begin{cor}\label{cor:Lseg-right-fibrations}
    If $p\colon X \to Y$ is a right fibration, then so is the Segalification $\bbL(p)\colon  \bbL(X) \to \bbL(Y)$.
\end{cor}
\begin{proof}
    The target map $t \colon \Ar(\calS) \to \calS$ is a cartesian fibration and thus, by the opposite of \cite[Example 5.2.8.15]{HTT},
    we have a factorization system on $\Ar(\calS)$ whose right part $\Ar(\calS)^\mrm{cart} \subseteq \Ar(\calS)$ consist of the cartesian edges, equivalently pullback squares. 
    (The left part consists of morphisms which induce equivalence on the target.)
    Note that $(p\colon X \to Y) \in \Ar(\PSh(\simp)) = \Fun(\Dop,\Ar(\calS))$ is right-$\Ar(\calS)^\mrm{cart}$-fibered if and only if it is a right fibration, 
    hence it suffices to verify the conditions of \cref{cor:right-fibered-usable} for $\calV = \Ar(\calS)$ equipped with the aforementioned factorization system.
    The first condition holds since colimits in $\Ar(\calS)$ are stable under base change. 
    The second condition holds since colimits in $\calS$ are stable under base change.
\end{proof}

\begin{rem}
    Examination of the proof of \cref{cor:Lseg-right-fibrations} shows that the same result holds if $\calS$ is replaced with any presentable \category{} $\calV$ in which sifted colimits are stable under base change.
\end{rem}

Recall that the nerve $\hldef{\xN_\bullet} \colon \Cat \to \PSh(\simp)$ is fully faithful and its essential image is precisely the complete Segal spaces.
We write $\hldef{\cat{-}} \colon \PSh(\simp) \to \Cat$ for the left adjoint of $\xN_\bullet$.
We let $\hldef{\Lcss}\colon \PSh(\simp) \to \PSh(\simp)$ denote the localization onto the complete Segal spaces.
With this notation we have for any $X \in \PSh(\simp)$ a canonical equivalence $\xN_\bullet \cat{X} \simeq \Lcss X$.
A model-categorical proof of the following proposition was given by Rasekh \cite[Theorem 4.18 and 5.1]{Rasekh17}.

\begin{prop}\label{prop:right-fibration-invariance}
    For any simplicial space $X$ there is an adjoint equivalence of \categories{}:
    \[\begin{tikzcd}
	{\Lseg \colon \PSh(\simp)^{\rm r-fib}_{/X}} && {(\Seg_{\Dop})^{\rm r-fib}_{/\Lseg(X)} \colon X \times_{\Lseg(X)}(-)}
	\arrow[""{name=0, anchor=center, inner sep=0}, shift left=2, from=1-1, to=1-3]
	\arrow[""{name=1, anchor=center, inner sep=0}, shift left=2, from=1-3, to=1-1]
	\arrow["\simeq"{description}, draw=none, from=0, to=1]
    \end{tikzcd}\]
\end{prop}
\begin{proof}
    Combining \cref{cor:Lseg-right-fibrations} and \cref{lem:equivalence-of-right-fibrations}, we learn that if $E \to X$ is a right fibration and $X$ is Segal, then $E$ is also Segal.
    We claim that $X \times_{\Lseg(X)} (-) \colon (\Seg_{\Dop})^{\rm r-fib}_{/\Lseg(X)}=\PSh(\,simp)^{\rm r-fib}_{/\Lseg(X)} \to  \PSh(\simp)^{\rm r-fib}_{/X}$ is right adjoint to the functor $\Lseg \colon \PSh(\simp)^{\rm r-fib}_{/X} \to  (\Seg_{\Dop})^{\rm r-fib}_{/\Lseg(X)}$ afforded by \cref{cor:Lseg-right-fibrations}.
    Indeed for any $(E \to X) \in \PSh(\simp)^{\rm r-fib}_{/X}$ and $(E' \to \Lseg(X)) \in (\Seg_{\Dop})^{\rm r-fib}_{/\Lseg(X)}$ we have
    \[\Map_{/X}(E,X \times_{\Lseg(X)}E') \simeq \Map_{/\Lseg(X)}(E,E') \simeq \Map_{/\Lseg(X)}(\Lseg(E),E').\]
    It remains to check that the unit and counit, given respectively by $E \to X \times_{\Lseg(X)} \Lseg(E)$ and $\Lseg(X \times_{\Lseg(X)} E') \to E'$, are equivalences. 
    Since $\Lseg$ does not affect the $0$-simplices, both maps evaluate to equivalences at $[0]$.
    The claim now follows from
    \cref{lem:equivalence-of-right-fibrations}.
\end{proof}

\begin{lem}\label{lem:equivalence-of-right-fibrations}
    Suppose $X \to Y$ and $X' \to Y$ are right fibrations and $f\colon X \to X'$ is a map over $Y$.
    Then $f$ is an equivalence if and only if $f_0\colon X_0 \to X_0'$ is.
\end{lem}
\begin{proof}
    Since $f \colon X' \to X$ is a map of right fibrations the map $f_n\colon X_n \to X_n'$ can be recovered as the base change
    of the map $f_0\colon X_0 \to X_0'$ over $Y_0$ along $(d_0)^n\colon Y_n \to Y_0$.
\end{proof}

\begin{cor}\label{cor:right-fibrations-old}
    For any simplicial space $X$ there is an adjoint equivalence of \categories{}:
    \[\begin{tikzcd}
	{\Lcss \colon \PSh(\simp)^{\rm r-fib}_{/X}} && {(\CSeg_{\Dop})^{\rm r-fib}_{/\Lcss(X)} \colon X \times_{\Lcss(X)}(-)}
	\arrow[""{name=0, anchor=center, inner sep=0}, shift left=2, from=1-1, to=1-3]
	\arrow[""{name=1, anchor=center, inner sep=0}, shift left=2, from=1-3, to=1-1]
	\arrow["\simeq"{description}, draw=none, from=0, to=1]
    \end{tikzcd}\] 
\end{cor}
\begin{proof}
    From \cref{cor:Lseg-right-fibrations} and \cite[Proposition A.21]{HK22} we learn that right fibrations are preserved by $\Lseg \colon \PSh(\simp) \to \Seg_{\Dop}$ and $\Lcpl \colon \Seg_{\Dop} \to \CSeg_{\Dop}$ respectively so their composite yields a functor
    \[
        \Lcss \colon \PSh(\simp)_{/X}^{\rm r-fib} 
        \xrightarrow{\Lseg}
        \PSh(\simp)_{/\Lseg X}^{\rm r-fib} 
        \xrightarrow{\Lcpl}
        \PSh(\simp)_{/\Lcss X}^{\rm r-fib}.
    \]
    The first is an equivalence by \cref{prop:right-fibration-invariance} and the second by \cite[Proposition A.22]{HK22}.
\end{proof}

Under the equivalence $\xN_\bullet \colon \Cat \simeq \CSeg_\Dop(\calS)$ 
the functor $\cat{-}$ is identified with $\Lcss \simeq  \Lcpl \Lseg$ from which we learn the following:
\begin{cor}\label{cor:right-fibrations}
    The functor $\mbm{C}\colon \PSh(\simp) \to \Cat$ induces
    for any simplicial space $X$ an equivalence
    \[
        \PSh(\simp)_{/X}^{\rm r-fib} \xrightarrow[\simeq]{\mbm{C}(-)} 
        \Catover{\cat{X}}^{\rm r-fib}
        \simeq
        \PSh(\cat{X})
    \] 
\end{cor}

\subsubsection{A formula for $\cat{-}$}
Let $X$ be a simplicial space and write $\hldef{\simp_{/X}}$ for its \hldef{\category{} of simplices},
i.e.~the codomain of the associated right fibration $p_X \colon \simp_{/X} \to \simp$.
\cref{cor:right-fibrations} can be used to give  a formula for $\cat{X}$ as a certain localization of $\simp_{/X}$.
To do so we will need the ``last vertex map'' 
$\xN_\bullet(\simp_{/X}) \to X$ (see e.g.~\cite[\S 4]{HK22})
and the ``last vertex functor'' 
$e \colon \simp_{/X} \to \cat{X}$
obtained by applying $\cat{-}$ to it.
Let us write $\hldef{\simp_{\max}} \subseteq \simp$ for the wide subcategory spanned by morphisms which preserve the maximal element.

\begin{prop}
    The "last vertex functor" $e \colon \simp_{/X} \to \cat{X}$
    induces an equivalence of \categories{}
    \[\simp_{/X}[W_X^{-1}] \simeq \cat{X}.\]
    where $W_X \coloneq p_X^{-1}(\simp_{\max}) \subseteq \simp_{/X}$.
\end{prop}
\begin{proof}
    First we observe that every object in $\simp_{/X}$ is connected through a zig-zag of last vertex maps to a $0$-simplex and thus the maps
    $\simp_{/\Delta^0}[W_{\Delta^0}^{-1}] \to \simp_{/X}[W_X^{-1}]$
    induced by $\Delta^0 \to X$ are jointly essentially surjective.
    By \cref{lem:hack} it now suffices to construct an equivalence $\PSh(\simp_{/X}[W_X^{-1}]) \simeq \PSh(\cat{X})$ naturally in $X$.
    Right fibrations over $X$ are precisely $\Delta_{\rm max}$-\textit{equifibered} simplicial spaces over $X$ in the sense of \cite[Lemma 4.1.8]{properads}, and thus we have $\PSh(\simp_{/X}[W_X^{-1}]) \simeq \PSh(\simp)_{/X}^{\rm rfib}$.
    Combining with \cref{cor:right-fibrations} yields the desired equivalence $ \PSh(\simp_{/X}[W_X^{-1}]) \simeq \PSh(\cat{X})$.\qedhere
\end{proof}

\begin{lem}\label{lem:hack}
    Let $F\colon \PSh(\simp) \to \Cat$ be a functor such that 
    \begin{enumerate}[(1)]
        \item there is an equivalence $\PSh(F(-)) \simeq \PSh(\cat{-})$ of functors $\PSh(\simp) \to \PrL$, and 
        \item for all $X$ the functors $F(\Delta^0) \to F(X)$ induced by $x \colon \Delta^0 \to X$ are jointly surjective.
    \end{enumerate}
    Then $F$ is naturally equivalent to $\cat{-}$.
    Moreover, any natural transformation $F \Rightarrow \cat{-}$ is an equivalence.
\end{lem}
\begin{proof}
    The idempotent completion $\calC^{\rm idem}$ of any \category{} $\calC$ can be constructed as the full subcategory of atomic\footnote{An object $c \in \calC$ is called \hldef{atomic} (or completely compact in \cite{HTT}) if the copresheaf $\Map_\calC(c,-)$ preserves colimits.} objects in $\PSh(\calC)$ \cite[Proposition 5.1.6.8]{HTT}. 
    Restricting the equivalence from (1) to atomic objects yields a natural equivalence
    $\alpha_X\colon \cat{X}^{\rm idem} \simeq F(X)^{\rm idem}$
    of functors $\PSh(\simp) \to \Cat$.
    In particular we have $* \simeq \cat{\Delta^0}^{\rm idem} \simeq F(\Delta^0)^{\rm idem}$. 

    The joint essential image of the functors $* = \cat{\Delta^0}^{\rm idem} \to \cat{X}^{\rm idem}$ induced from the maps $\Delta^0 \to X$ is precisely $\cat{X} \subset \cat{X}^{\rm idem}$,
    and by (2) the joint essential image of $* = F(\Delta^0)^{\rm idem} \to F(X)^{\rm idem}$ is $F(X) \subset F(X)^{\rm idem}$.
    Since $\alpha$ is a natural equivalence, both it and its inverse must preserve these subcategories and thus $F(-) \simeq \cat{-}$.

    It remains to show every endomorphism of $\cat{-}$ is an equivalence.
    Since $\cat{-}$ is the left Kan extension of the inclusion $\simp \hookrightarrow \Cat$ along the Yoneda embedding $\simp \hookrightarrow \PSh(\simp)$, it suffices to observe that $\Id_{\simp}$ admits no non-trivial endomorphisms.
\end{proof}

\subsection{Segalification for \texorpdfstring{$(\infty,n)$-categories}{n-categories}}

By iterating the Segalification formula one can also obtain formulas for the Segalification for $(\infty,n)$-categories.
We recall the definition of $n$-fold Segal spaces due to Barwick \cite{Barwick-nfold}.
See \cite[Definition 2.2 and 2.4]{CS19} and \cite{haugseng2018equivalence} for a reference.

\begin{defn}
    An $n$-fold simplicial space $X \colon \simp^{\op, n} \to \calS$ is
    \begin{itemize}
        \item  \hldef{reduced} if for every $k \ge 0$ and $m_1,\dots,m_k \in \mbb{N}$ the $(n-k-1)$-fold simplicial space $X_{m_1,\dots,m_k,0,\bullet,\dots,\bullet}$ is constant.
        We denote by $\PSh^{r}(\simp^{\times n})$ the full subcategory spanned by reduced objects.
        \item an \hldef{$n$-uple Segal space} 
        if it is Segal in each coordinate,
        that is, if for every $k \ge 0$ and $m_1,\dots,m_{k-1},  m_{k+1}, \dots, m_n \in \mbb{N}$ the simplicial space $X_{m_1,\dots,m_{k-1}, \bullet, m_{k+1},\dots,m_n}$ is a Segal space.
        \item an \hldef{$n$-fold Segal space} 
        if it is an $n$-tuple Segal space and reduced.
        We denote by $\Seg_{\Dop}^{\rm n-fold} \subset \PSh^{r}(\simp^{\times n})$ the full subcategory spanned by the $n$-fold Segal spaces.
    \end{itemize}
\end{defn}

As we briefly explained in the introduction, complete $n$-fold Segal spaces model $(\infty,n)$-categories.
We will not discuss the issue of completeness here,
but rather our goal will be to give a formula for the Segalification of reduced $n$-fold simplicial spaces.
For $1 \le j \le n$ we denote by $\hldef{\bbL_j} \colon \PSh(\simp^{\times n}) \to \PSh(\simp^{\times n})$ the Segalification functor in the $j$-th coordinate.

\begin{lem}\label{lem:sifted-colimit-of-Segal}
    Suppose $F\colon K \to \PSh(\simp)$ is a diagram of simplicial spaces such that $K$ is sifted, each $F(k)$ is a Segal space, and the diagram $F(-)_0\colon K \to \calS$ is constant.
    Then the colimit $\colim\limits_{k \in K} F(k)$ is a Segal space.
\end{lem}
\begin{proof}
    A simplicial space $X$ is Segal if and only if the canonical map $X_n \to X_0 \times_{(X_0 \times X_0)} (X_{n-1} \times X_1)$ is an equivalence for all $n\ge 2$.
    In the case of $X = \colim_{k \in K} F(k)$ we therefore want show that the outside rectangle in the following diagram is cartesian:
\[\begin{tikzcd}
	{\colim\limits_{k \in K} F(k)_n} & {\colim\limits_{k \in K} F(k)_{n-1} \times F(k)_1 } & {\colim\limits_{k \in K} F(k)_{n-1} \times \colim\limits_{k \in K} F(k)_1 } \\
	{\colim\limits_{k \in K} F(k)_0} & {\colim\limits_{k \in K} F(k)_0 \times F(k)_0} & {\colim\limits_{k \in K} F(k)_0 \times \colim\limits_{k \in K} F(k)_0}
	\arrow[from=1-1, to=2-1]
	\arrow[from=1-3, to=2-3]
	\arrow[from=1-1, to=1-2]
	\arrow["\simeq", from=1-2, to=1-3]
	\arrow["\Delta", from=2-1, to=2-2]
	\arrow["\simeq", from=2-2, to=2-3]
	\arrow[from=1-2, to=2-2]
\end{tikzcd}\]
    The right horizontal maps are equivalences because $K$ is sifted
    and hence it suffices to consider the left square.
    This square is a colimit of the cartesian squares that we have because $F(k)$ is Segal for all $k$.
    Moreover, the bottom row of the squares is a constant functor in $k$ by assumption.
    So it follows that the colimit of the square is still cartesian because colimits in $\calS$ are stable under base change.
    As observed above, this implies that $\colim_{k \in K} F(k)$ is Segal as claimed.
\end{proof}

\begin{lem}\label{lem:L_j-keeps-Segal}
    Let $X_{\bullet, \dots \bullet}$ be a reduced $n$-fold simplicial space, then
    \begin{enumerate}
        \item $(\bbL_j X)_{\bullet, \dots, \bullet}$ is reduced for all $j$, and 
        \item if $X_{\bullet, \dots, \bullet}$ satisfies the Segal condition in the first $(j-1)$-coordinates, 
        then $(\bbL_j X)_{\bullet, \dots, \bullet}$ satisfies the Segal condition in the first $j$ coordinates.
    \end{enumerate}
\end{lem}
\begin{proof}
    Claim (1):
    We need to check that $(\bbL_j X)_{m_1, \dots, m_k, 0, \bullet, \dots, \bullet}$ is a constant simplicial space.
    For $k<j-1$ this is true because constant simplicial spaces are Segal and hence Segalifying in the $j$th coordinate does not change $X_{m_1, \dots, m_k, 0, \bullet, \dots, \bullet}$.
    For $k=j-1$ this is true because Segalifying never changes the $0$-simplices.
    For $k\ge j$ consider the simplicial object 
    $Y\colon \Dop \to \PSh(\simp^{\times \{k+2, \dots, n\}})$
    defined by sending $l$ to $X_{m_1, \dots, m_{j-1},l,m_{j+1}, \dots, m_k, 0, \bullet, \dots \bullet}$.
    By assumption $Y_l$ is a constant $(n-k-1)$-fold simplicial space for all $l$.
    Since the full subcategory of those $(n-k-1)$-fold simplicial spaces that are constant is closed under all limits and colimits,
    it follows that $(\Lseg Y)_l$ is still constant for all $l$.

    Claim (2):
    To simplify notation, we will assume that $n=2=j$;
    the general case is analogous.
    Suppose that $X_{\bullet, \bullet}$ satisfies the Segal condition in the first coordinate.
    It suffices to show that $\bbL_2 X_{\bullet, \bullet}$ still satisfies the Segal condition in the first coordinate.
    In other words, we need to show that $(\bbL_2 X)_{\bullet, l}$ is a Segal space for all $l$.
    By the Segal condition in the second coordinate, it suffices to do so for $l=0, 1$.
    For $l=0$ there is nothing to show since Segalification does not change the $0$-simplices.
    For $l=1$ we have the necklace formula
    \[
        (\bbL_2 X)_{\bullet, 1} \simeq 
        \colim_{\Delta^{n_1}\vee \dots \vee \Delta^{n_k} \in \Nec^\op}  X_{\bullet, n_1} \times_{X_{\bullet, 0}} \dots \times_{X_{\bullet, 0}} X_{\bullet, n_k}
    \]
    where the pullbacks and colimit are computed in simplicial spaces.
    To complete the proof it suffices to show that the diagram on the right hand side, whose colimit we are taking, satisfies the hypotheses of \cref{lem:sifted-colimit-of-Segal}.
    The indexing category $\Nec^\op$ is sifted by \cref{lem:nec-sifted} and each of the terms in the diagram is a Segal space since Segal spaces are closed under pullbacks.
    It remains to observe that the diagram of $0$-simplices is constant, as its value on any necklace $N = \Delta^{n_1} \vee \dots \vee \Delta^{n_k}$ is
    \[
        X_{0, n_1} \times_{X_{0, 0}} \dots \times_{X_{0, 0}} X_{0, n_k}
        \simeq X_{0, 0} \times_{X_{0, 0}} \dots \times_{X_{0, 0}} X_{0, 0}
        \simeq X_{0,0}
    \]
    by the reduced assumption.
\end{proof}

\begin{prop}\label{prop:iterated-segal-formula}
    The left adjoint to the full inclusion of $n$-fold Segal space into reduced $n$-fold simplicial spaces may be computed as
    \[
       \bbL = \bbL_n \circ \dots \circ \bbL_1 \colon
       \Seg_{\Dop}^{n\rm -fold} \adj \PSh^{r}(\simp^{\times n}) 
       :\! \mrm{inc}
    \]
\end{prop}
\begin{proof}
    For any $n$-fold simplicial space the map
    $Y_{\bullet, \dots, \bullet} \to (\bbL_j Y)_{\bullet,\dots,\bullet}$
    is local with respect to those $n$-fold simplicial spaces that are Segal in the $j$th coordinate.
    Therefore, for any $n$-fold simplicial space $X_{\bullet, \dots, \bullet}$
    all of the maps
    \[
        X_{\bullet, \dots, \bullet} \to
        (\bbL_1 X)_{\bullet, \dots, \bullet} \to 
        (\bbL_2 \circ \bbL_1)(X)_{\bullet, \dots, \bullet} \to 
        \dots \to
        (\bbL_n \circ\dots \circ \bbL_1)(X)_{\bullet, \dots, \bullet} 
    \]
    are local with respect to the full subcategory of $n$-tuple Segal spaces.
    If we assume that $X_{\bullet, \dots, \bullet}$ is also reduced, then it follows by inductively applying \cref{lem:L_j-keeps-Segal} that $(\bbL_j \circ \dots \circ \bbL_1)(X)_{\bullet, \dots, \bullet}$ is reduced and satisfies the Segal condition in the first $j$ coordinates.
    We have therefore shown that the map 
    \[
        X_{\bullet, \dots, \bullet} \to
        (\bbL_n \circ\dots \circ \bbL_1)(X)_{\bullet, \dots, \bullet} 
    \]
    is local with respect to $n$-fold Segal spaces and that its target is an $n$-fold Segal space.
    Consequently, it exhibits the target as the localization onto the full subcategory $\Seg_{\Dop}^{\rm n-fold}$.
\end{proof}

\begin{war}
    In the context of \cref{prop:iterated-segal-formula}
    the order in which the Segalification functors are applied is crucial. 
    It is important to Segalify the $1$-morphisms first, then the $2$-morphisms, and so on.
    If we were to apply $\bbL_2$ first and then $\bbL_1$, 
    it would no longer clear that the result is $\bbL_2$-local as, $\bbL_1$ can break the Segal condition in the second simplicial direction.
\end{war}

\subsection{The Boardman-Vogt tensor product}

In this section we show how to use the Segalification formula to give a new construction
of the Boardman-Vogt tensor product of \operads{}. We begin with a brief recollection on the tensor product of commutative monoids. 

\subsubsection{Recollection on tensor product of commutative monoids}

For an \category{} with products $\calC$ we let $\hldef{\CMon(\calC)} \subset \Fun(\Fin_*, \calC)$ denote the \category{} of commutative monoids in $\calC$, see \cite[\S 1]{gepner-universality}.
In the case $\calC=\calS$ we simply write $\hldef{\CMon} \coloneq \CMon(\calS)$.
We let $\hldef{\SM} \coloneq \CMon(\Cat)$ denote the \category{} of symmetric monoidal \categories{}.
By applying $\CMon(-)$ to the adjunction $\cat{-} \dashv \xN_\bullet$
and using that $\CMon(\PSh(\simp)) \simeq \Fun(\Dop, \CMon)$,
we obtain an adjunction:
\[
    \cat{-}\colon \Fun(\Dop, \CMon)  \adj \SM :\! \xN_\bullet.
\]
The right adjoint here is the \hldef{symmetric monoidal nerve}, which in the $n$th level is given by $\xN_n(\calD) = \Fun([n], \calD)^\simeq$ with the pointwise symmetric monoidal structure.

Let $\calC$ be a cartesian closed presentable \category{}.
Gepner, Groth and Nikolaus \cite{gepner-universality} show that $\CMon(\calC)$ admits a unique symmetric monoidal structure $\otimes$ such that the free functor $\xF\colon \calC \to \CMon(\calC)$ is symmetric monoidal.
Moreover, if $\calC$ and $\calD$ are presentable cartesian closed and $L \colon \calC \to \calD$ is a symmetric monoidal (i.e.~finite product preserving) left adjoint, they show the induced functor $L\colon \CMon(\calC) \to \CMon(\calD)$ is canonically symmetric monoidal \cite[Lemma 6.3.(ii)]{gepner-universality}.
Segalification, completion, and $\cat{-}$ are examples of such $L$. 
Using that $\Seg_{\Dop}(\CMon) \simeq \CMon(\Seg_{\Dop}(\calS))$ (and similarly for complete Segal spaces),
we record this for future use.

\begin{cor}\label{cor:Lseg-and-Lc-symmetric-monoidal}
    All of the functors in the following commutative diagram are canonically symmetric monoidal for the respective tensor product:
    \[\begin{tikzcd}
	{\Fun(\Dop,\CMon) } & {\Seg_\Dop(\CMon)} & {\CSeg_\Dop(\CMon)} & \SM
	\arrow["{\Lcpl}", from=1-2, to=1-3]
	\arrow["{\Lseg}", from=1-1, to=1-2]
	\arrow["{\cat{-}}", shift right=1, curve={height=24pt}, from=1-1, to=1-4]
	\arrow["\simeq", from=1-4, to=1-3]
	\arrow["{\xN_\bullet(-)}"', draw=none, from=1-4, to=1-3]
    \end{tikzcd}\]
\end{cor}

While the characterisation of the tensor product in \cite{gepner-universality} is an excellent tool for studying the symmetric monoidal structure on $\SM$ as a whole, we will also need a more ``local'' description that gives a universal property for the tensor product of two fixed symmetric monoidal \categories{}.
Below, in \cref{prop:tensor=localised-Day}, we give such a description by closely following \cite[\S 4.3]{shay-HigherMonoids}.

\subsubsection{Some equifibered theory}
A morphism of commutative monoids $f \colon M \to N$ is said to be \hldef{equifibered} if the canonical square
\[\begin{tikzcd}
	{M\times M} & M \\
	{N \times N} & N
	\arrow["f \times f"', from=1-1, to=2-1]
	\arrow["{+}", from=1-1, to=1-2]
	\arrow["{+}", from=2-1, to=2-2]
	\arrow["f",from=1-2, to=2-2]
\end{tikzcd}\]
is cartesian \cite[Definition 2.1.4]{properads}. 
Equifibered maps span a replete subcategory of commutative monoids $\hldef{\CMon^\eqf} \subset \CMon$. 
This notion was introduced in op.~cit.~for the purpose of developing the theory of \properads{}.
A quintessential feature of equifibered maps is that a morphism of free monoids $f \colon \xF(X) \to \xF(Y)$ is equifibered if and only if it is free, i.e.~$f \simeq \xF(g)$ for some map of spaces $g \colon X \to Y$.
Equifibered maps can be thought of as a well-behaved generalization of free maps; for example they form the right part of a factorization system on $\CMon$. 
Further details can be found in \cite[\S 2]{properads}.

\begin{obs}\label{obs:tensor-equifibered}
    Equifibered maps between free monoids are closed under the tensor product. 
    Indeed, the free functor $\xF \colon \calS \to \CMon$ is symmetric monoidal and by \cite[Remark 2.1.7]{properads} induces an equivalence onto the subcategory $\CMon^{\free,\eqf} \subseteq \CMon$ of free monoids and equifibered maps.
\end{obs}

\begin{defn}
    A simplicial commutative monoid $M \colon \Dop \to \CMon$ is called \hldef{$\otimes$-disjunctive} if it is right-$\CMon^\eqf$-fibered.
\end{defn}

\begin{rem}
    By \cite[Lemma 3.2.15]{properads} the nerve $\xN_\bullet \calC$ of a symmetric monoidal \category{} $\calC \in \SM$ is $\otimes$-disjunctive, 
    if and only if $\calC$ is $\otimes$-disjunctive in the sense of \cite[Definition 3.2.14]{properads}.
    That is, if and only if for all $x, y \in \calC$ the functor
    \[
        \otimes\colon \calC_{/x} \times \calC_{/y} \too \calC_{/x \otimes y}
    \]
    is an equivalence.
\end{rem}

\begin{obs}\label{lem:free-tensor-disjunctive}
    Let $M \colon \Dop \to \CMon$ be $\otimes$-disjunctive. 
    Then $M$ is level-wise free if and only if $M_0$ is free. 
    Indeed, evaluation at the last vertex $d_0 \circ \dots \circ d_0  \colon M_n \to M_0$ is equifibered so, if $M_0$ is free, the same holds for $M_n$ \cite[Corollary 2.1.11]{properads}.
\end{obs}

As a consequence of the necklace formula we have the following:
\begin{lem}\label{lem:Segalify-tensor-disjunctive}
    Let $M \in \Fun(\Dop,\CMon)$ be $\otimes$-disjunctive. 
    Then $\Lseg (M)$ is $\otimes$-disjunctive. 
\end{lem}
\begin{proof}
    It suffices to check the conditions of \cref{cor:right-fibered-usable}.
    The first condition was verified in \cref{ex:O-algebra-universal-colimits}.
    The second condition follows from \cite[Lemma 2.1.22]{properads}.
\end{proof}

We now give a description of \operads{} using equifibered maps.

\begin{defn}\label{defn:preoperads}
    A \hldef{pre-operad} is a Segal commutative monoid $M\in \Seg_\Dop(\CMon)$ which is $\otimes$-disjunctive and level-wise free. 
    A pre-operad is called \hldef{complete} if its underlying Segal space is.
\end{defn}

\begin{war}
    Pre-operads in the sense of \cref{defn:preoperads} should not be confused with $\infty$-preoperads in the sense of Lurie \cite[\S 2.1.4]{HA}.
    Instead, in the language of \cite{properads}, a pre-operad is precisely a monic pre-properad.
\end{war}

Pre-operads are to \operads{} what Segal spaces are to \categories{}.
Indeed the envelope functor induces an equivalence between Lurie's \operads{} and complete pre-operads.

\begin{thm}\label{thm:Env-equivalence}
    Lurie's monoidal envelope functor
    \(
        \Env(-)\colon \Op \to \SM
    \)
    is faithful (i.e.~it induces a monomorphism on mapping spaces).
    Moreover, the composite
    \[
        \Op \xrightarrow{\Env(-)} \SM 
        \stackrel{\xN_\bullet}{\simeq} 
        \CSeg(\CMon)
    \]
    identifies $\Op$ with the (non-full) subcategory of $\CSeg(\CMon)$ whose objects are complete pre-operads and whose morphisms are equifibered natural transformations.
\end{thm}

The first instance of this theorem can be found in the work of Haugseng--Kock \cite{HK21}, who showed that the sliced functor $\Env\colon \Op \to \SMover{\Fin}$ is fully faithful and characterised its image.
Barkan--Haugseng--Steinebrunner \cite{envelopes} then gave an alternative characterization of the image, closely related to pre-operads.
The above formulation was given in \cite[Theorem 3.2.13]{properads}.

\subsubsection{Tensor products of \texorpdfstring{$\infty$-operads}{infinity operads}}

We can now apply the necklace formula for Segalification to show that pre-operads are closed under the tensor product.

\begin{prop}\label{prop:tensor-of-pre-operads}
    The replete subcategory $\pOp \subseteq \Seg_\Dop(\CMon)$ is closed under the tensor product. 
\end{prop}
\begin{proof}
    First we claim that if $M \colon \Dop \to \CMon$ is level-wise free and $\otimes$-disjunctive then the same holds for $\Lseg M$. 
    Indeed \cref{lem:Segalify-tensor-disjunctive} shows that $\Lseg(M)$ is $\otimes$-disjunctive and since $\Lseg(M)_0 \simeq M_0$ is free the claim follows from \cref{lem:free-tensor-disjunctive}.
    
    To complete the proof it suffices to show that if $M, N \in \Fun(\Dop,\CMon)$ are $\otimes$-disjunctive and level-wise free, then the same holds for their tensor product $M\otimes N$.
    This follows from the fact that equifibered maps between free monoids are closed under the tensor product (\cref{obs:tensor-equifibered}).
\end{proof}

We are now in a position to prove \cref{thmA:BV-tensor}.

\begin{proof}[Proof of \cref{thmA:BV-tensor}]
    For the first part it suffices by \cref{thm:Env-equivalence} to show that complete pre-operads are closed under the tensor product.
    By \cref{cor:Lseg-and-Lc-symmetric-monoidal}
    the equivalence $\SM \simeq \CSeg_\Dop(\CMon)$ identifies the tensor product of symmetric monoidal \categories{} with the following bi-functor on complete Segal monoids
    \[(M_\bullet,N_\bullet) \longmapsto \Lcss(M_\bullet \otimes N_\bullet) \simeq \Lcpl \Lseg(M_\bullet \otimes N_\bullet).\]
    Suppose now that $M_\bullet$ and $N_\bullet$ are pre-operads.
    By \cref{prop:tensor-of-pre-operads}, $\Lseg(M_\bullet \otimes N_\bullet)$ is a pre-operad and hence by \cite[Proposition 3.4.7]{properads} so is the completion $\Lcpl \Lseg(M_\bullet \otimes N_\bullet)$.
    
    For the second part we must compare $\otimes_{\rm BV}$ to Lurie's tensor product.
    It follows from \cref{lem:Lurie-tensor-and-Env} below that there is an equivalence
    $\Env(\calO \otimes_\Lurie \calP) \simeq \Env(\calO \otimes_{\rm BV} \calP)$
    for all \operads{} $\calO$ and $\calP$.
    And since $\Env$ is an equivalence onto a replete subcategory by \cref{thm:Env-equivalence}, it follows that we already have such an equivalence before applying $\Env$.
\end{proof}

\subsubsection{Comparison to Lurie's Boardman--Vogt tensor product}
To complete \cref{thmA:BV-tensor} we need to compare Lurie's Boardman--Vogt tensor product to the tensor product of symmetric monoidal \categories{}.
Lurie defines the tensor product of two \operads{} $\calO \to \Fin_*$ and $\calP \to \Fin_*$ to be the universal \operad{} representing bifunctors of \operads{}.
A bifunctor of \operads{} \cite[Definition 2.2.5.3]{HA} is a functor $F \colon \calO \times \calP \to \calQ$ together with a square
\[\begin{tikzcd}
    \calO \times \calP \ar[d] \ar[r, dashed, "F"] &
    \calQ \ar[d] \\
    \Fin_* \times \Fin_* \ar[r, "\wedge"] & 
    \Fin_*
\end{tikzcd}\]
such that $F$ preserves cocartesian lifts of inert morphisms.
We begin by giving an analogous characterisation of the tensor product of symmetric monoidal \categories{}.

Let $\calV \in \PrL$ be a cartesian closed presentable \category, such as $\Cat$.
The smash product of finite pointed sets $A_+ \wedge B_+ = (A \times B)_+$ induces via Day convolution \cite[\S 3]{shay-HigherMonoids} a symmetric monoidal structure on the \category{} of functors $\Fun(\Fin_*, \calV)$.
Writing $\mu\coloneq \wedge \colon \Fin_* \times \Fin_* \to \Fin_*$ for the smash product functor, we can describe the Day convolution as
$F \boxtimes_{\rm Day} G = \mu_!(F \times G).$
This tensor product localises to a symmetric monoidal structure on the full subcategory of commutative monoids $\CMon(\calV) \subset \Fun(\Fin_*, \calV)$:
\begin{prop}\label{prop:tensor=localised-Day}
    The left adjoint in the localisation adjunction
    \[
        \bbL\colon \Fun(\Fin_*, \calV) \adj \CMon(\calV)
    \]
    admits a symmetric monoidal structure with respect to the Day convolution on the functor category and the tensor product of commutative monoids on the right.
\end{prop}
\begin{proof}
    First we argue that the Day convolution symmetric monoidal structure localises to the full subcategory $\CMon(\calV)$,
    i.e.~that there is a (unique) symmetric monoidal structure on $\CMon(\calV)$ for which $\bbL$ is symmetric monoidal.
    This follows by essentially the same argument as \cite[Proposition 4.24]{shay-HigherMonoids}, except that we need to check the analogue of the second part of \cite[Lemma 4.22]{shay-HigherMonoids}.
    Indeed, if $X \colon \Fin_\ast \to \calV$ is a commutative monoid and $A_+ \in \Fin_*$ then $X(A_+ \wedge -)$ is still a commutative monoid as can be seen by
    \[\begin{tikzcd}
	{ X(A_+ \wedge B_+) } & {\displaystyle\prod_{b \in B} X(A_+ \wedge \{b\}_+)} & {   \displaystyle \prod_{(a,b) \in A \times B} X(\{a\}_+ \wedge \{b\}_+).}
	\arrow["\simeq", curve={height=-18pt}, from=1-1, to=1-3]
	\arrow[from=1-1, to=1-2]
	\arrow["\simeq", from=1-2, to=1-3]
    \end{tikzcd}\]
    Now proceed as in \cite[Theorem 4.26]{shay-HigherMonoids} to argue that the localised Day convolution symmetric monoidal structure on $\CMon(\calV)$ satisfies the universal property of \cite[Theorem 5.1]{gepner-universality}.
\end{proof}

Given a symmetric monoidal \category{} $\calC \in \SM$ we let $\calC^\otimes \coloneq \Un(\calC \colon \Fin_\ast \to \Cat) \to \Fin_\ast$ denote its associated cocartesian fibration.
With this notation $\Map_{\SM}(\calC,\calD) = \Map_{/\Fin_\ast}^{\rm cocart}(\calC^\otimes,\calD^\otimes)$ where the latter denotes the subspace of $\Map_{/\Fin_\ast}(\calC^\otimes,\calD^\otimes)$ spanned by functors which preserve cocartesian edges. 
(In \cite{HA} this is taken as a definition.) 
Similarly, we have $\Map_{\Op}(\calO, \calP) = \Map_{/\Fin_\ast}^{\rm int-cocart}(\calO, \calP)$.
In this setting, the tensor product admits the following characterization:

\begin{cor}
    If $\calC^\otimes$ and $\calD^\otimes$ are (unstraightened) symmetric monoidal \categories{}, then there is an equivalence
    \[
        \Map_{/\Fin_*}^{\rm cocart}(\calC^\otimes \otimes \calD^\otimes, \calE^\otimes) 
        \simeq \Map_{/\Fin_* \times \Fin_*}^{\rm cocart}(\calC^\otimes \times \calD^\otimes, \mu^* \calE^\otimes)
    \]
    natural in the symmetric monoidal \category{} $\calE^\otimes$.
\end{cor}

This is entirely analogous to the universal property of the Boardman--Vogt tensor product of \operads{} in \cite[Definition 2.2.5.3 and Remark 2.2.5.4]{HA}, which in this language may be stated as:
\[
    \Map_{/\Fin_*}^{\rm int-cocart}(\calO \otimes \calP, \calQ) 
    \simeq \Map_{/\Fin_* \times \Fin_*}^{\rm int-cocart}(\calO \times \calP, \mu^* \calQ).
\]

\begin{lem}\label{lem:Lurie-tensor-and-Env}
     Lurie's tensor product satisfies that for any two \operads{} $\calO$ and $\calP$ there is a canonical equivalence:
    \[
    \Env(\calO \otimes_{\Lurie} \calP) \simeq
    \Env(\calO) \otimes \Env(\calP) 
    \]
\end{lem}
The proof of \cref{lem:Lurie-tensor-and-Env} will use a variant of Lurie's symmetric monoidal envelope construction in which $\Fin_\ast$ is replaced by $\Fin_\ast \times \Fin_\ast$. 
More precisely, let $\calQ$ be an \category{} with a factorisation system $(\calQ^\xint, \calQ^\act)$ and let $p\colon \calO \to \calQ$ be a functor with cocartesian lifts for inerts. 
Then the envelope of $p$ is defined as:
\[
    \Env_\calQ(p\colon \calO \to \calQ) \coloneq \calO \times_{\calQ} \Ar^\act(\calQ),
\]
where $\Ar^\act(\calQ) \subset \Ar(\calQ)$ denotes the full subcategory of the arrow category spanned by active morphisms. 
This functor was studied extensively in \cite[\S 3]{shah2022parametrized} and \cite[\S 2]{envelopes}.
In particular, for any cocartesian fibration $\calE \to \calQ$,  \cite[Proposition 2.2.4]{envelopes} provides an equivalence: 
\[ \Fun_{/Q}^{\rm int-cocart}(\calO, \calE) \simeq \Fun_{/Q}^{\rm cocart}(  \Env_\calQ(p\colon \calO \to \calQ) ,\calE). \tag{$\star$} \label{eq:envelope}\]
\begin{proof}
    As before, let $\mu \colon \Fin_\ast \times \Fin_\ast \to \Fin_\ast$ denote the smash product.
    For any $\calC \in \SM$ we have
    \begin{align*}
        \Fun^\otimes(\Env(\calO \otimes_{\rm Lurie}\calP), \calC) & \simeq \Alg_{\calO \otimes_{\Lurie} \calP}(\calC)\\
        &\simeq \Fun_{/(\Fin_* \times \Fin_*)}^{\rm int-cocart}( \calO \times \calP, \mu^* \calC)  \\
        &\simeq \Fun_{/(\Fin_* \times \Fin_*)}^{\rm cocart}( \Env_{\Fin_* \times \Fin_*}(\calO \times \calP), \mu^* \calC) \qquad \qquad \qquad \text{\hyperref[eq:envelope]{$(\star)$}}  \\
        &\simeq \Fun_{/(\Fin_* \times \Fin_*)}^{\rm cocart}( \Env_{\Fin_*}(\calO) \times \Env_{\Fin_*}(\calP), \mu^* \calC) \\ 
        &\simeq \Fun_{/(\Fin_* \times \Fin_*)}^{\rm cocart}( \Env(\calO) \times \Env(\calP), \mu^* \calC) \\
        & \simeq \Fun^\otimes(\Env(\calO) \otimes \Env(\calP), \calC),
    \end{align*}
    where the fourth equivalence uses the identification $\Ar^\act(\Fin_\ast \times \Fin_\ast) \simeq \Ar^\act(\Fin_\ast) \times \Ar^\act(\Fin_\ast)$.
\end{proof}

\printbibliography[heading=bibintoc]

\end{document}